\documentclass[10pt]{amsart}
\usepackage{enumerate,amsmath,amssymb,latexsym,
amsfonts, amsthm, amscd, MnSymbol}

\theoremstyle{plain}

\newtheorem{theorem}{Theorem}

\numberwithin{equation}{section}

\newcommand{\R}{\mathbb{R}}
\newcommand{\C}{\mathbb{C}}

\newcommand{\sm}{\rm sm}
\newcommand{\cm}{\rm cm}

\newcommand{\ii}{\rm i}

\addtocounter{theorem}{-1}

\begin{document}

\title {Meromorphic functions in a first-order system}

\date{}

\author[P.L. Robinson]{P.L. Robinson}

\address{Department of Mathematics \\ University of Florida \\ Gainesville FL 32611  USA }

\email[]{paulr@ufl.edu}

 \keywords{}

\begin{abstract}

We consider the differential equations $s' = c^5$ and $c' = - s^5$ with initial conditions $s(0) = 0$ and $c(0) = 1$. We show that $s^2 c^2$ extends to the reciprocal of the Weierstrass elliptic function with invariants $g_2 = 0$ and $g_3 = 16$.  We also show that powers of $s^2 c^2$ are the only functions of the form $s^m c^n$ that extend meromorphically in the plane. 
\end{abstract}

\maketitle

\medbreak

\section*{Introduction} 

\medbreak 

To each integer $p$ greater than one, we associate the initial value problem 
$$s\,' = c^{p - 1}, \, c\,' = - s^{p - 1}; \: s(0) = 0, \, c(0) = 1.$$
A solution pair to this IVP on a connected neighbourhood of the origin is immediately seen to satisfy a generalization of the Pythagorean identity: 
$$s^p + c^p = 1.$$

\medbreak 

The first IVP in this sequence is of course very familiar: if $p = 2$ then $s = \sin$ and $c = \cos$. Beyond the observation that these solutions are entire and simply periodic, we make no further comment on this case. 

\medbreak 

The next IVP in the sequence is decidedly less familiar: if $p = 3$ then the solutions $s = \sm$ and $c = \cm$ are meromomorphic and doubly periodic: they are the Dixonian elliptic functions, named in honour of Dixon [1] who actually investigated a whole family of third-order elliptic functions including them. 

\medbreak 

We considered the case $p = 4$ in [4]. There, we proved that neither $s$ nor $c$ is elliptic but that the quadratic expressions $s c$, $s^2$ and $c^2$ are elliptic: in fact, we identified them as rational expressions in the lemniscatic Weierstrass function and its derivative. We remark that after [4] was completed, we found that it has nontrivial overlap with [3]. 

\medbreak 

Here, we consider the case $p = 6$. First among our findings is the fact that the quartic expression $s^2 c^2$ is elliptic, with simple poles: indeed, we find that it is the reciprocal of the (scaled equianharmonic) Weierstrass function with invariants $g_2 = 0$ and $g_3 = 16$. This finding goes towards our proof that $s^2 c^2$ and its powers are the only functions of the form $s^m c^n$ that extend meromorphically in the plane. As some relief, we also establish positive results regarding simply periodic meromorphic extensions to bands centred on the real axis. 

\medbreak 

In retrospect, the approach taken here may be applied to the case $p = 4$ considered in [4]: if $s$ and $c$ denote the solutions of the initial value problem in that case, then $s^m c^n$ does not extend to the plane meromorphically unless $m + n$ is even, when it actually extends to an elliptic function. Looking ahead, in the case $p = 5$ that we have temporarily passed over, the corresponding function $s^m c^n$ is only elliptic when $m = n = 0$, because of the five-fold symmetry present in that case. We shall report on such further matters in due course. 

\medbreak 

\section*{Meromorphicity and Periodicity} 

\medbreak 

We shall study the initial value problem ({\bf IVP}): 
$$s\,' = c^5, \, c\,' = - s^5 ; \; s(0) = 0, \, c(0) = 1$$ 
and begin with the most elementary observations regarding it and its solutions. 

\medbreak 

First of all, it is convenient to note the following analogue of the `Pythagorean' identity. 

\medbreak 

\begin{theorem} \label{Pyth} 
Any solution pair to {\bf IVP} on a connected open neighbourhood of the origin satisfies 
$$s^6 + c^6 = 1.$$ 
\end{theorem} 

\begin{proof} 
The differential equations in {\bf IVP} show that $(s^6 + c^6)\,' = 0$ so that $s^6 + c^6$ is constant; the initial conditions in {\bf IVP} show that $s^6 + c^6$ equals $1$ at the origin and hence throughout its connected domain. 
\end{proof} 

\medbreak 

Note that {\bf IVP} does have exactly one pair $(s, c)$ of holomorphic solutions in a suitably small neighbourhood of the origin; this is an elementary application of the classical Picard existence-uniqueness theorem, as follows. 

\medbreak 

\begin{theorem} \label{Picard} 
The initial value problem {\bf IVP} 
$$s\,' = c^5, \, c\,' = - s^5 ; \; s(0) = 0, \, c(0) = 1$$ 
has a unique holomorphic solution pair $(s, c)$ in the open disc $B_r (0)$ of radius $r = 4^4/5^5.$ 
\end{theorem} 

\begin{proof} 
We use the Picard theorem essentially as it appears in Section 2.3 of [2]. With $b > 0$ fixed, if $| s | \leqslant b$ and $|c - 1 | \leqslant b$ then $| s^5 | \leqslant b^5$ and $| c^5 | \leqslant (b + 1)^5$. The (autonomous) Picard theorem therefore guarantees that {\bf IVP} has a unique holomorphic solution pair in the disc about the origin of radius $b/(b + 1)^5$. Taking $b \to 1/4$ maximizes this. 
\end{proof} 

\medbreak 

We remark that we could increase the radius of the disc by instead first solving for $s$ the initial value problem  
$$s\,' = (1 - s^6)^{5/6} ; \; s (0) = 0$$
and then setting $c = (1 - s^6)^{1/6}$; here, Theorem \ref{Pyth} is involved and the powers assume their principal values. 

\medbreak 

Our primary interest is not in local properties of solutions to {\bf IVP} but rather in their global properties. Nevertheless, it is convenient to pause at this point and establish certain symmetries enjoyed by the solution pair $(s, c)$ in the disc $B_r (0)$. If we are able to extend $s$ and $c$ beyond $B_r (0)$ to suitably symmetric connected open neighbourhoods of the origin, the principle of analytic continuation guarantees that these symmetries will continue to be enjoyed by the extensions. 

\medbreak 

The solution pair $(s, c)$ to {\bf IVP} has a six-fold rotational symmetry. Let 
$$\gamma = e^{\ii \pi /3} = \tfrac{1}{2} (1 + \ii \sqrt{3})$$
so that $\gamma^3 = - 1$ and $\gamma$ is a primitive sixth root of unity. 

\medbreak 

\begin{theorem} \label{six} 
If $z \in B_r (0)$ then $s (\gamma z) = \gamma \, s(z)$ and $c( \gamma z) = c(z)$. 
\end{theorem} 

\begin{proof} 
Define functions $S$ and $C$ in $B_r (0)$ by $S(z) = \overline{\gamma} \, s(\gamma z)$ and $C(z) = c( \gamma z)$. By direct calculation, $S$ and $C$ satisfy {\bf IVP}; by Theorem \ref{Picard}, they coincide with $s$ and $c$ respectively.  
\end{proof} 

\medbreak

Similar consideration of $\overline{s(\overline{z})}$ and $\overline{c(\overline{z})}$ as functions of $z$ shows that $s$ and $c$ are `real' in the sense that $s(\overline{z}) = \overline{s(z)}$ and $c(\overline{z}) = \overline{s(z)}$. Threefold application of Theorem \ref{six} shows that $s$ is odd and $c$ is even. 

\medbreak 

Further, let 
$$\delta =  e^{\ii \pi /6} = \tfrac{1}{2} (\sqrt{3} + \ii)$$
so that $\delta^2 = \gamma$ and $\delta^6 = -1$. Define functions $f$ and $g$ by 
$$f(z) = \overline{\delta} s(\delta z) \; \; \; {\rm and} \; \; \; g(z) = s(\delta z)$$ 
for all $z$ in the disc $B_r (0)$. In the present context, this construction serves as a counterpart to the passage from circular functions to hyperbolic functions. 

\medbreak 

\begin{theorem} \label{hyp} 
The functions $f$ and $g$ satisfy the initial value problem 
$$f\,' = g^5, \, g\,' = f^5 ; \; f(0) = 0, \, g(0) = 1$$ 
and the identity 
$$g^6 - f^6 = 1.$$ 
\end{theorem} 

\begin{proof} 
The initial value problem comes from {\bf IVP} directly; the identity comes from Theorem \ref{Pyth} by substitution or by mimicking its proof. 
\end{proof} 

\medbreak 

Multiplication by $\gamma$ transforms the functions $f$ and $g$ as it transformed $s$ and $c$; similarly, $f$ is odd and $g$ is even. Moreover, $f$ and $g$ are `real' in the same sense that $s$ and $c$ are `real'; in particular, $f$ is real-valued on the real interval $(- r, r) = \R \cap B_r (0)$. Consequently, $s$ maps the line segment $(-r, r) \, \delta$ to itself. Were this consequence alone required, it would follow from Theorem \ref{six} and the ensuing discussion, for 
$$(-r, r) \, \delta = \{ z \in B_r (0) : z = \gamma \overline{z} \}$$ 
and if $x \in (-r, r)$ then 
$$s(\delta x) = s(\gamma \overline{\delta} x) = \gamma s(\overline{\delta} x) = \gamma \overline{s(\delta x)}.$$ 

\medbreak 

Now, our fundamental concern is the question whether or not the functions $s$ and $c$ (more generally, products of their powers) extend to functions that are meromorphic in the plane. Let us assume that the solution pair $(s, c)$ has been extended beyond $B_r (0)$ to a connected open neighbourhood of the origin. A pole for one function in the pair $(s, c)$ is a pole for the other; say its order is $m$ for $s$ and $n$ for $c$. From {\bf IVP} it follows that $m + 1 = 5 n$ and $n + 1 = 5 m$ and therefore that $m = n = 1/4$. This elementary observation shows that $s$ and $c$ themselves cannot be extended meromorphically (unless they are actually extended holomorphically); however, it leaves open the possibility that at least one of the fourth-degree functions $s^4, \, s^3 c, \, s^2 c^2, \, s c^3, \, c^4$ does so extend. As we now proceed to demonstrate, the function $s^2 c^2$ admits extension to a function that is meromorphic in the plane, indeed is elliptic; as we more than demonstrate subsequently, the remaining four quartics lack meromorphic extensions to the plane. 

\medbreak 

The affirmative case is readily settled, in fact with an explicit identification of the elliptic continuation. For the quartic 
$$q = s^2 c^2$$ 
we have 
$$q\,' = 2 s s\,' c^2 + 2 s^2 c c\,' = 2 s c^7 - 2 s^7 c = 2 s c( c^6 - s^6)$$
so that 
$$(q\,')^2 = 4 s^2 c^2 \{ (c^6 + s^6)^2 - 4 s^6 c^6 \}$$ 
whence by Theorem \ref{Pyth} 
$$(q\,')^2 = 4 q \, (1 - 4 q^3).$$ 

\medbreak 

\begin{theorem} \label{elliptic} 
The quartic $s^2 c^2$ extends to the elliptic function $1/\wp$ where $\wp = \wp (-\, ; \, 0, 16)$ is the Weierstrass function with invariants $g_2 = 0$ and $g_3 = 16$. 
\end{theorem} 

\begin{proof} 
With $p = 1/q$ we have $q\,' = q^2 p\,'$ so that 
$$q^4 (p\,')^2 = 4 q \, (1 - 4 q^3)$$ 
and therefore 
$$(p\,')^2 = 4 p^3 - 16.$$ 
The solutions to this first-order differential equation are the translates of $\wp (-\, ; \,  0, 16)$. The proof is concluded by noting that $p = 1/q$ has a double pole at the origin, since $q = s^2 c^2$ has a double zero there. 
\end{proof} 

\medbreak 

By a convenient and harmless abuse, we may identify $(s^2 c^2)^{-1}$ with the Weierstrass function $\wp (-\, ; \, 0, 16)$. In future, we may similarly conflate a function and one of its extensions, even when a notational constituent of the function is not extended. 

\medbreak 

Building upon the foundation of this affirmative case, we shall now demonstrate that the `pure' quartics $s^4$ and $c^4$ do not extend to functions that are meromorphic in the plane. In fact, it is convenient to go further: we shall show that neither $s^{12}$ nor $c^{12}$ so extends. 

\medbreak 

\begin{theorem} \label{pure}
Neither $s^{12}$ nor $c^{12}$ extends to a function that is meromorphic in the plane. 
\end{theorem} 

\begin{proof} 
We again press the capitalized symbols $S$ and $C$ into temporary service, this time defining $S = s^{12}$ and $C = c^{12}$. From Theorem \ref{Pyth} and Theorem \ref{elliptic} we deduce that 
$$\wp^{-3} = s^6 c^6 = s^6 (1 - s^6) = s^6 - s^{12}$$ 
so that 
$$S + \wp^{-3} = s^6$$ 
and therefore 
$$(S + \wp^{-3})^2 = S.$$ 
Rearrangement yields 
$$S^2 + (2 \wp^{-3} - 1) S + \wp^{-6} = 0$$ 
or 
$$(S + \wp^{-3} - \tfrac{1}{2})^2 = \frac{\wp^3 - 4}{4 \wp^3}$$ 
upon completing the square. Recalling from within Theorem \ref{elliptic} the differential equation satisfied by the Weierstrass function, there follows 
$$16 \, (\wp\,')^{-2} \, (S + \wp^{-3} - \tfrac{1}{2})^2 = \wp^{-3}.$$
Now, suppose that $S = s^{12}$ extends to a function that is meromorphic in the plane; the quadratic identity just established continues to hold. However: the function on the left has poles of even order, because it is a square; the function on the right has poles of odd order, because the zeros of $\wp$ are simple. This contradiction faults the supposition that $s^{12}$ extends meromorphically in the plane. The argument that $C = c^{12}$ does not extend is parallel: indeed, the symmetry of Theorem \ref{Pyth} and Theorem \ref{elliptic} in $s$ and $c$ ensures that $C$ also satisfies the identity $(C + \wp^{-3})^2 = C$. 
\end{proof} 

\medbreak 

We leave as an exercise the similar (but slightly more elaborate) verification that neither $s^{24}$ nor $c^{24}$ admits meromorphic extension in the plane. It will be found that if $E$ denotes either of these functions then  
$$4 \, (\wp\,')^{-2} \, (E - \wp^{-6} - \tfrac{1}{8} \wp^{-3} (\wp\,')^2)^2 = (1 + \tfrac{1}{16} \wp^3 (\wp\,')^2) \, \wp^{-9}.$$

\medbreak 

We may also settle (negatively) the question whether either of the remaining quartics extends meromorphically in the plane. 

\medbreak 

\begin{theorem} \label{31+13} 
Neither $s^3 c$ nor $s c^3$ extends to a function that is meromorphic in the plane. 
\end{theorem} 

\begin{proof} 
If $s^3 c$ or $s c^3$ were to extend meromorphically then so would $s^6 c^2$ or $s^2 c^6$; by Theorem \ref{elliptic} it would then follow that $s^4$ or $c^4$ extends, contrary to Theorem \ref{pure}. 
\end{proof} 

Further, the quotient $s/c$ is not the restriction to $B_r (0)$ of a function that is meromorphic in the plane: if it were, then so would be $s^4 = (s^2 c^2)(s/c)^2$; but this is contrary to fact. 

\medbreak 

A closer inspection of the quotient $s/c$ will actually enable us to determine precisely which products of the form $s^m c^n$ are the restrictions of functions that are meromorphic in the plane. The presentation to come will supersede some of what has gone before; we retain Theorem \ref{pure} for the sake of its proof. 

\medbreak 

\begin{theorem} \label{t}
The quotient $t = s/c$ satisfies the differential equation 
$$(t\,')^3 = 1 + t^6$$
and the initial condition 
$$t\,' (0) = 1.$$ 
\end{theorem} 

\begin{proof} 
The differential equation rests on Theorem \ref{Pyth}: 
$$\big(\frac{s}{c}\big)\,' = \frac{c s\,' - s c\,'}{c^2} = \frac{c^6 + s^6}{c^2} = \frac{1}{c^2}$$
so 
$$\big(\frac{s}{c}\big)\,'^{\,\,3} = \frac{1}{c^6} = \frac{c^6 + s^6}{c^6} = 1 + \big(\frac{s}{c}\big)^6.$$
The initial condition is obvious; note that the stated initial condition implies the condition $t(0) = 0$ but not vice versa. 
\end{proof} 

\medbreak 

Let us subject the quotient $t = s/c$ to the same transformation that was used in Theorem \ref{hyp}: for $z \in B_r (0)$ define 
$$h(z) = \overline{\delta} \, t(\delta z) = \frac{\overline{\delta} \, s(\delta z)}{c(\delta z)} = \frac{f(z)}{g(z)}.$$
It follows that the quotient $h = f/g$ satisfies the differential equation 
$$(h\,')^3 = 1 - h^6$$
together with the initial conditions 
$$h\,' (0) = 1 \; \; {\rm and} \; \; h(0) = 0.$$ 
By integration, it follows that for $z$ in a sufficiently small disc about the origin
$$z = \int_0^{h(z)} \frac{{\rm d}\zeta}{(1 - \zeta^6)^{1/3}}$$
where the cube root in the integrand has its principal value. 

\medbreak 

This puts us in contact with classical Schwarz-Christoffel theory. Recall that the rule 
$$w \mapsto \int_0^w \frac{{\rm d}\zeta}{(1 - \zeta^6)^{1/3}}$$
(with principal-valued cube root) maps the open unit disc $\mathbb{D}$ conformally to the open regular hexagon $\mathbb{H}$ with $0$ as its centre and with the real point 
$$L = \int_0^1 \frac{{\rm d}\xi}{(1 - \xi^6)^{1/3}}$$ 
as one of its vertices, the full set of vertices being 
$$\{ L, L \gamma, - L \overline{\gamma}, - L, - L \gamma, L \overline{\gamma} \}; $$
further, this rule maps the closed unit disc $\overline{\mathbb{D}}$ to the closed regular hexagon $\overline{\mathbb{H}}$ homeomorphically. 

\medbreak 

It follows that the quotient $h = f/g$ extends as the inverse of this Schwarz-Christoffel transformation: it maps the open regular hexagon $\mathbb{H}$ to the open unit disc $\mathbb{D}$ conformally, the map being continuous up to the boundary; in particular, it maps $L$ to $1$. 

\medbreak 

Reversing the transformation that led from $t = s/c$ to $h = f/g$ rotates $\mathbb{H}$ through angle $\pi/6$. Thus, $t = s/c$ admits an extension that maps the open hexagon $\delta \, \mathbb{H}$ to $\mathbb{D}$ conformally, continuously up to the boundary. Note that the hexagon $\delta \, \mathbb{H}$ has sides of length $L$ and that its boundary meets the real axis at the points $\pm K$ where 
$$K = \int_0^1 \frac{{\rm d}\xi}{(1 + \xi^6)^{1/3}}= \tfrac{1}{2} \sqrt{3} \,  L\, ; $$
the first of these expressions for $K$ records the time taken for $t$ to increase from zero to unity while the second follows from the geometry of the regular hexagon. 

\medbreak 

Let us write $\mathbb{H}_0 = \delta \, \mathbb{H}$. For each integer $n$, translation through $2 n K$ produces a congruent hexagon $\mathbb{H}_n = 2 n K + \mathbb{H}_0$. Write $\mathcal{H}$ for the union of these hexagons along with the open edges that adjacent hexagons share. 

\medbreak 

\begin{theorem} \label{band}
The quotient $t = s/c$ extends to a function that is meromorphic in $\mathcal{H}$; this extension has real period $4 K$ and is holomorphic except for simple poles at the odd multiples of $2 K$. 
\end{theorem} 

\begin{proof} 
As the values of $t$ on the boundary of $\mathbb{H}_0 = \delta \, \mathbb{H}$ lie in the unit circle, $t$ extends merorphically across the vertical edges of $\mathbb{H}_0$ by Schwarz reflexion, the simple zero at $0$ producing simple poles at $\pm 2 K$. This extends $t$ from the hexagon $\mathbb{H}_0$ to the two adjacent hexagons $\mathbb{H}_{\pm 1}$; repeated reflexion extends $t$ across the whole array of hexagons. To see that $t$ has $4 K$ as period, let $z_0 \in \mathbb{H}_0$ have images $z_{-1} \in \mathbb{H}_{-1}$ and $z_1 \in \mathbb{H}_1$ under reflexion in the vertical edges of $\mathbb{H}_0$: on the one hand, $z_1 = z_{-1} + 4 K$ by inspection; on the other hand, Schwarz symmetry yields  
$$t(z_1) = 1/\overline{t(z_0)} = t(z_{-1}).$$ 
\end{proof} 

\medbreak 

Note that $t$ further extends continuously up to the boundary of $\mathcal{H}$ and that its values on the boundary have unit modulus. The value of $t$ at each hexagonal vertex is a sixth root of $- 1$; in particular, $t(\delta \, L) = \delta$. 

\medbreak 

It follows that $t$ is a simply periodic meromorphic function in the horizontal band 
$$\{ z \in \C: 2 \, | {\rm Im} \, z | < L\}$$
of width $L$ centred on the real axis. We remark that $t$ does not extend meromorphically to any wider band: in fact, we claim that the vertices of the hexagons prohibit such extension. As $t$ is not constant and $t^6 = -1$ at the vertices of each hexagon, our claim is substantiated by the following lemma. 

\medbreak 

\begin{theorem} \label{holdup}
Let $t$ be holomorphic and satisfy $(t')^3 = 1 + t^6$ in the connected open set $U$. If $t'$ vanishes at some point of $U$ then $t$ is constant throughout $U$. 
\end{theorem} 

\begin{proof} 
By differentiation, $3 \, (t')^2 \, t'' = 6 \, t^5 \, t'$ so that 
$$t' \, (t' \, t'' - 2 \, t^5) = 0$$
throughout $U$. As $U$ is connected, either $t\,' \equiv 0$ or $t' \, t'' - 2 t^5 \equiv 0$. However, if $a \in U$ has $t' (a) = 0$ then $t(a) \neq 0$: thus the latter alternative fails, so the former alternative holds; this forces $t$ to be constant throughout $U$. 
\end{proof} 

\medbreak 

In particular, the vertices of the hexagons prevent extension of $t = s/c$ from $\mathcal{H}$ to a function that is meromorphic in the plane. We can say more. 

\medbreak 

\begin{theorem} \label{tpower}
If $\ell$ is a nonzero integer, then $t^{\ell} = (s/c)^{\ell}$ does not extend to a function that is meromorphic in the plane. 
\end{theorem} 

\begin{proof} 
Without loss, take $\ell$ to be positive. Theorem \ref{band} provides a unique meromorphic extension of $t^{\ell}$ to the band $\mathcal{H}$ that is continuous up to the boundary, where it takes values in the unit circle. Now, let $\mathbb{A}$ be the translate of the hexagon $\mathbb{H}_0 = \delta \, \mathbb{H}$ with centre $a = \delta \, L + \ii \, L$: the regular hexagon $\mathbb{A}$ shares an edge with each of $\mathbb{H}_0$ and $\mathbb{H}_1$; all three hexagons share $\delta \, L$ as a vertex. Schwarz symmetry precludes enlarging the domain of $t^{\ell}$ to include $\mathbb{A}$: on the one hand, reflexion across the edge shared by $\mathbb{A}$ and $\mathbb{H}_0$ would convert the zero of $t^{\ell}$ at $0$ to a pole of $t^{\ell}$ at $a$; on the other hand, reflexion across the edge shared by $\mathbb{A}$ and $\mathbb{H}_1$ would convert the pole of $t^{\ell}$ at $2 K$ to a zero of $t^{\ell}$ at $a$
\end{proof} 

\medbreak 

We are now in a position to provide the advertised determination of precisely which products $s^m c^n$ extend meromorphically in the plane. 

\medbreak 

\begin{theorem} \label{extend}
If $m$ and $n$ are integers, then $s^m c^n$ extends to a function that is meromorphic in the plane precisely when $m$ and $n$ are equal and even. 
\end{theorem} 

\begin{proof} 
If $k$ is an integer then $s^{2 k} c^{2 k}$ extends to the elliptic function $\wp^{- k}$ according to Theorem \ref{elliptic}. In the opposite direction, suppose that $s^m c^n$ extends meromorphically to the plane. The discussion prior to Theorem \ref{elliptic} shows that the total degree $m + n$ of $s^m c^n$ must be a multiple of $4$; say $m + n = 4 k$. It follows that $m - n$ must be even; write $m - n = 2 \ell$. Now the factorization 
$$s^m c^n = (s^2 c^2)^{k} \, (s/c)^{\ell}$$ 
or equivalently  
$$(s/c)^{\ell} = (s^2 c^2)^{-k} \, (s^m c^n)$$
implies that $(s/c)^{\ell}$ extends meromorphically to the plane. Finally, Theorem \ref{tpower} forces $\ell = 0$ and we conclude that $m = n = 2k$. 
\end{proof} 

\medbreak 

Thus, the only products $s^m c^n$ that extend meromorphically to the plane actually extend as elliptic functions: namely, as powers of the Weierstrass function $\wp( - \, ; 0, 16)$. 

\medbreak 

Finally, we end on a more positive note: each of the products $s^m c^n$ for which $m + n$ is a multiple of $4$ extends to the band $\mathcal{H}$ of hexagons as a meromorphic function that is simply periodic. That each extends meromorphically is evident from the fact that $s^2 c^2$ and $s/c$ so extend. To show that each meromorphic extension is simply periodic and to determine its period,  we compare the least positive period of $\wp$ in Theorem \ref{elliptic} and the least positive period of $t$ in Theorem \ref{band}.

\medbreak 

The least positive period of $\wp$ is $2 \omega$ where $\omega$ is the least positive zero of $\wp'$. From the differential equation 
$$(\wp\,')^2 = 4 \wp^3 - 16$$ 
it follows that $\wp(\omega) = 4^{1/3}$ and 
$$\omega = -  \int_0^{4^{1/3}} \frac{{\rm d}\xi}{(4 \xi^3 - 16)^{1/2}}. $$
The substitution $u = 4 \xi^{-3}$ converts this to a Beta integral, which may be evaluated in terms of the Gamma function: 
$$\omega = \frac{1}{3 \cdot 4^{1/3}} \int_0^1 u^{-5/6} (1 - u)^{-1/2} {\rm d} u = \frac{1}{6 \cdot 2^{1/3}} \frac{\Gamma(\tfrac{1}{6}) \Gamma(\tfrac{1}{2})}{\Gamma(\tfrac{2}{3})}\,.$$

\medbreak 

The least positive period of $t$ is $4 K$ where 
$$K = \int_0^1 \frac{{\rm d}\xi}{(1 + \xi^6)^{1/3}};$$
we may evaluate this in terms of the Gamma function as follows. By inversion, 
$$\int_0^1 \frac{{\rm d}\xi}{(1 + \xi^6)^{1/3}} = \int_1^{\infty} \frac{{\rm d}\xi}{(1 + \xi^6)^{1/3}}= \frac{1}{2} \, \int_0^{\infty} \frac{{\rm d}\xi}{(1 + \xi^6)^{1/3}} $$ 
and by the substitution $u = (1 + \xi^6)^{-1}$ we deduce that  
$$\int_0^{\infty} \frac{{\rm d}\xi}{(1 + \xi^6)^{1/3}} = \frac{1}{6} \int_0^1 u^{-5/6} (1 - u)^{-5/6} {\rm d} u = \frac{1}{6} \frac{\Gamma(\tfrac{1}{6})^2}{\Gamma(\tfrac{1}{3})} \, .$$
Thus 
$$K = \frac{1}{12} \frac{\Gamma(\tfrac{1}{6})^2}{\Gamma(\tfrac{1}{3})} \, .$$

\medbreak 

The Gamma duplication formula allows us to relate the numbers $\omega$ and $K$: in fact, the duplication formula 
$$2 \Gamma(\tfrac{1}{2}) \Gamma (2 z) = 2^{2 z} \Gamma(z) \Gamma(z + \tfrac{1}{2})$$
with $z = 1/6$ yields at once their equality: 
$$\omega = K.$$ 
\medbreak 

\begin{theorem} \label{scband}
If $m + n$ is a multiple of $4$ then $s^m c^n$ extends to be meromorphic and simply periodic in the band $\mathcal{H}$. 
\end{theorem} 

\begin{proof} 
We of course dismiss the trivial case $m = n = 0$. As noted previously, $m - n$ is even. Now, consider the factorization 
$$s^m c^n = (s^2 c^2)^{(m + n)/4} (s/c)^{(m - n)/2}$$
and refer to Theorem \ref{elliptic} and Theorem \ref{band}. Both $s^2 c^2$ and $s/c$ certainly extend to be meromorphic in $\mathcal{H}$; accordingly, so does $s^m c^n$. The function $s^2 c^2 = \wp^{-1}$ has period $2 \omega = 2 K$. The odd function $s/c = t$ has period $4 K$; its square $(s/c)^2$ has period $2 K$. Consequently, $s^m c^n$ has $2 K$ or $4 K$ as period according to whether the even integer $m - n$ is or is not a multiple of $4$; equivalently in the present context, according to whether $m$ and $n$ are both even or both odd. 
\end{proof} 

\medbreak 

\bigbreak

\begin{center} 
{\small R}{\footnotesize EFERENCES}
\end{center} 
\medbreak 

[1] A.C. Dixon, {\it On the doubly periodic functions arising out of the curve $x^3 + y^3 - 3 \alpha x y = 1$}, The Quarterly Journal of Pure and Applied Mathematics, 24 (1890) 167-233. 

\medbreak 

[2] E. Hille, {\it Ordinary Differential Equations in the Complex Domain}, Wiley-Interscience (1976); Dover Publications (1997).

\medbreak 

[3] A. Levin, {\it A Geometric Interpretation of an Infinite Product for the Lemniscate Constant}, The American Mathematical Monthly, Volume 113 Number 6 (2006) 510-520. 

\medbreak 

[4] P.L. Robinson, {\it The elliptic functions in a first-order system}, arXiv 1903.07147 (2019). 

\medbreak

\end{document}